\documentclass[12pt]{article}
\usepackage{amsthm}
\usepackage{times}
\usepackage{newtxmath}
\usepackage[margin=1.0in]{geometry}

\newtheorem{definition}{Definition}
\newtheorem{proposition}{Proposition}
\newtheorem{remark}{Remark}

\title{Arithmetic Hodge-Iwasawa Moduli Stacks}
\author{Xin Tong}
\date{}

\begin{document}

\maketitle

\begin{abstract}
\noindent In this article, we are going to construct arithmetic moduli stacks of $G$-bundles after our previous construction on Hodge-Iwasawa theory. These stacks parametrize certain Hodge-Iwasawa structures in a coherent way.
\end{abstract}

\tableofcontents

\newpage
\section{Introduction}

We discuss in this article moduli stacks of $G$-isocrystals over arithmetic families of Fargues-Fontaine curves after the constructions in our previous consideration in \cite{T1} and \cite{T2}. The initial motivation after \cite{Iwa}, \cite{T}, \cite{F} is a combination reasonable of the Iwasawa deformation after Burns-Flach-Fukaya-Kato \cite{BF1}, \cite{BF2}, \cite{FK} and some arithmetic deformation after Kedlaya-Pottharst-Xiao and Kedlaya-Pottharst \cite{KPX}, \cite{KP}, of the structures after \cite{KL1} and \cite{KL2}. When we consider the integration effect of the deformation we will expect some sort of moduli stacks in arithmetic situation, which is the analogy of those considered in \cite{He1}, \cite{PR}, \cite{Dr}, \cite{SW}, \cite{Ked}, \cite{RZ}, \cite{G}, \cite{Dr1}, \cite{Dr2}, \cite{Har}, \cite{Har1},  \cite{HV}, \cite{M}, \cite{FS}, \cite{Sch1}.

\section{Arithmetic Moduli Stacks}

\subsection{Arithmetic Moduli Stacks of Filtered Frobenius Bundles}

\indent We now consider arithmetic moduli stacks of vector bundles after \cite{Dr}, \cite{Dr1}, \cite{Dr2}, \cite{He1}, \cite{PR}, \cite{EG}, \cite{EGH}, \cite{HHS} in the foundation in \cite{HP}, \cite{FF}, \cite{KL1} and \cite{KL2} with generalization initialized in \cite{KP}, \cite{T1} and \cite{T2}. Let $K$ be a complete adic field with valuation which we will assume to be discrete. We use the notation $k$ to denote the residue field of $K$ which we assume to be overall perfect in our consideration in this article. We then have the tower $*_{\mathrm{toric}}$ toric attached to this field. Recall we have the Fargues-Fontaine curve $\mathrm{FF}_{K,\mathrm{toric}}$ attached to $K$ in \cite{KL1} and \cite{KL2} which is defined to be Frobenius quotient of the space:
\begin{align}
\bigcup_{I=[s,r]}\mathrm{Spa}(\widetilde{\Pi}^I_{\mathrm{toric}},\widetilde{\Pi}_{\mathrm{toric}}^{I,+}).	
\end{align}

Then after \cite{W}, \cite{He1} we consider the sites carrying ffqc topology (faithfully flat and quasicompact topology):
\begin{align}
\mathrm{RigidAn}_{\mathbb{Q}_p,\mathrm{ffqc}},\mathrm{RigidAn}_{\mathbb{F}_p((t)),\mathrm{ffqc}}	
\end{align}
which consist of all the rigid analytic spaces. There are also animated versions of these categories as in \cite{CS1}, \cite{CS2} and \cite{CS3}, while in this derived setting we mainly consider the affinoid algebras which are defined to be the animation of usual rigid analytic affinoids in \cite{Ta}. In notation we use the notations as in the following to denote the opposite $(\infty,1)$-categories:
\begin{align}
\mathrm{AnimaRigidAn}^\mathrm{aff}_{\mathbb{Q}_p,\mathrm{ffqc}},\mathrm{AnimaRigidAn}^\mathrm{aff}_{\mathbb{F}_p((t)),\mathrm{ffqc}}.	
\end{align}

\begin{definition}
We now define the $(\infty,1)$-prestack $\mathrm{Finiteproj}_{\mathrm{FF}_{K,\mathrm{toric}}}(-)$ fibered over 
\begin{align}
\mathrm{AnimaRigidAn}^\mathrm{aff}_{\mathbb{Q}_p,\mathrm{ffqc}}
\end{align}
to be the functor sending $X$ to $(\infty,1)$-groupoid of vector bundles over $\mathrm{FF}_{K,\mathrm{toric},X}$ being regarded as Clausen-Scholze analytic space in \cite{CS1}, \cite{CS2} and \cite{CS3}. One then has the definition for equal characteristic situation. When specializing a filtration $\mathrm{Fil}$ with respect to some cocharacter of the group scheme $\mathrm{GL}$, we can define the prestack $\mathrm{Finiteproj}^{\mathrm{Fil}}_{\mathrm{FF}_{K,\mathrm{toric}}}(-)$.
\end{definition}

\begin{definition}
We now define the $(\infty,1)$-prestack $\mathrm{Finiteproj}_{\mathrm{FF}^\mathrm{imper}_{K,\mathrm{toric}}}(-)$ fibered over 
\begin{align}
\mathrm{AnimaRigidAn}^\mathrm{aff}_{\mathbb{Q}_p,\mathrm{ffqc}}
\end{align}
to be the functor sending $X$ to $(\infty,1)$-groupoid of vector bundles over $\mathrm{FF}^\mathrm{imper}_{K,\mathrm{toric},X}$. The imperfect space we consider here is defined to be the Frobenius quotient of
\begin{align}
\bigcup_{I=[s,r]}\mathrm{Spec}^\mathrm{CS}({\Pi}^I_{\mathrm{toric}}\otimes^\blacksquare \mathcal{O}_X,{\Pi}_{\mathrm{toric}}^{I,+}\otimes^\blacksquare \mathcal{O}_X).	
\end{align}
One then has the definition for equal characteristic situation. When specializing a filtration $\mathrm{Fil}$ with respect to some cocharacter of the group scheme $\mathrm{GL}$, we can define the prestack $\mathrm{Finiteproj}^{\mathrm{Fil}}_{\mathrm{FF}^\mathrm{imper}_{K,\mathrm{toric}}}(-)$.
\end{definition}

\begin{definition}
We now define the prestack $\mathrm{Finiteproj}^\mathrm{disc}_{\mathrm{FF}_{K,\mathrm{toric}}}(-)$ fibered over $\mathrm{RigidAn}_{\mathbb{Q}_p,\mathrm{ffqc}}$ to be the functor sending $X$ to groupoid of vector bundles over $\mathrm{FF}_{K,\mathrm{toric},X}$. One then has the definition for equal characteristic situation. When specializing a filtration $\mathrm{Fil}$ with respect to some cocharacter of the group scheme $\mathrm{GL}$, we can define the prestack $\mathrm{Finiteproj}^{\mathrm{disc},\mathrm{Fil}}_{\mathrm{FF}_{K,\mathrm{toric}}}(-)$.
\end{definition}

\begin{definition}
We now define the prestack $\mathrm{Finiteproj}^\mathrm{disc}_{\mathrm{FF}^\mathrm{imper}_{K,\mathrm{toric}}}(-)$ fibered over $\mathrm{RigidAn}_{\mathbb{Q}_p,\mathrm{ffqc}}$ to be the functor sending $X$ to groupoid of vector bundles over $\mathrm{FF}^\mathrm{imper}_{K,\mathrm{toric},X}$. The imperfect space we consider here is defined to be the Frobenius quotient of
\begin{align}
\bigcup_{I=[s,r]}\mathrm{Spec}^\mathrm{CS}({\Pi}^I_{\mathrm{toric}}\otimes^\blacksquare \mathcal{O}_X,{\Pi}_{\mathrm{toric}}^{I,+}\otimes^\blacksquare \mathcal{O}_X).	
\end{align}
One then has the definition for equal characteristic situation. When specializing a filtration $\mathrm{Fil}$ with respect to some cocharacter of the group scheme $\mathrm{GL}$, we can define the prestack $\mathrm{Finiteproj}^{\mathrm{disc},\mathrm{Fil}}_{\mathrm{FF}^\mathrm{imper}_{K,\mathrm{toric}}}(-)$.
\end{definition}

\begin{proposition}
The prestack $\mathrm{Finiteproj}^{\mathrm{disc},\mathrm{Fil}}_{\mathrm{FF}^\mathrm{imper}_{K,\mathrm{toric}}}(-)$ fibered over $\mathrm{RigidAn}_{\mathbb{Q}_p,\mathrm{ffqc}}$ is an Artin stack. The parallel result in equal characteristic sitaution holds true as well.	
\end{proposition}

\begin{proof}
Our construction could be regarded as relative version of the construction in \cite[Section 4.1]{He1}. 
\begin{align}
\bigcup_{I=[s,r]}\mathrm{Res}_{\Pi^I_\mathrm{toric}/\mathbb{Q}_p}\mathrm{GL}^\mathrm{rigidification}_{\Pi^I_\mathrm{toric}}\times \left(\mathrm{Res}_{\Pi^I_\mathrm{toric}/\mathbb{Q}_p}\mathrm{GL}^\mathrm{rigidification}_{\Pi^I_\mathrm{toric},E}/\mathrm{Parab}^\mathrm{rigidification}_\mathrm{Fil}\right)\\
/\bigcup_{I=[s,r]}\mathrm{Res}_{\Pi^I_\mathrm{toric}/\mathbb{Q}_p}\mathrm{GL}^\mathrm{rigidification}_{\Pi^I_\mathrm{toric},E}
\end{align}
reflects the certain stackification with $E/\mathbb{Q}_p$ a suitable extension.
\end{proof}

\indent We then consider the quotient in the $(\infty,1)$-categories to derive the corresponding Artin $(\infty,1)$-stacks generalizing the previous discussion above.

\begin{proposition}
The $(\infty,1)$-prestack $\mathrm{Finiteproj}^{\mathrm{Fil}}_{\mathrm{FF}^\mathrm{imper}_{K,\mathrm{toric}}}(-)$ fibered over $\mathrm{AnimaRigidAn}^\mathrm{aff}_{\mathbb{Q}_p,\mathrm{ffqc}}$ is an $(\infty,1)$-stack. The parallel result in equal characteristic sitaution holds true as well.	
\end{proposition}
 
\begin{proof}
For instance in the mixed-characteristic situation, by the results above we now form the quotient:
\begin{align}
\bigcup_{I=[s,r]}\mathrm{Res}_{\Pi^I_\mathrm{toric}/\mathbb{Q}_p}\mathrm{GL}^\mathrm{rigidification}_{\Pi^I_\mathrm{toric}}\times \left(\mathrm{Res}_{\Pi^I_\mathrm{toric}/\mathbb{Q}_p}\mathrm{GL}^\mathrm{rigidification}_{\Pi^I_\mathrm{toric},E}/\mathrm{Parab}^\mathrm{rigidification}_\mathrm{Fil}\right)\\
/^\mathrm{der}\bigcup_{I=[s,r]}\mathrm{Res}_{\Pi^I_\mathrm{toric}/\mathbb{Q}_p}\mathrm{GL}^\mathrm{rigidification}_{\Pi^I_\mathrm{toric},E}
\end{align}	
in the derived sense, which gives the result.
\end{proof}

\begin{remark}
These $(\infty,1)$-stacks actually admit presentations in some smooth sense in the derived sense. With the same proof one can have the corresponding representability of the prestacks in the equal characteristic situations, which we can omit the details.
\end{remark}

\subsection{Arithmetic Hecke Stacks, Arithmetic Moduli Stacks of Shtukas}

Now we consider the analogues of the Hecke stacks and moduli of shtukas in \cite{Laff}, \cite{Dr}, \cite{Dr1}, \cite{Dr2}, \cite{SW}, \cite{RZ}, \cite{G}, \cite{A}, \cite{La} in our arithmetic deformation situation in the following. First we can drop the filtration in the previous discussion to define arithmetic moduli of vector bundles, namely with the notations above we have the prestacks over the spaces:
\begin{align}
\widetilde{Y}:=\bigcup_{I=[s,r]}\mathrm{Spa}(\widetilde{\Pi}^I_{\mathrm{toric}},\widetilde{\Pi}_{\mathrm{toric}}^{I,+}), Y:=\bigcup_{I=[s,r]}\mathrm{Spa}({\Pi}^I_{\mathrm{toric}},{\Pi}_{\mathrm{toric}}^{I,+}). 	
\end{align}

\begin{definition}
We now define the $(\infty,1)$-prestack $\mathrm{Finiteproj}_{\widetilde{Y}}(-)$ fibered over $\mathrm{AnimaRigidAn}^\mathrm{aff}_{\mathbb{Q}_p,\mathrm{ffqc}}$ to be the functor sending $X$ to $(\infty,1)$-groupoid of vector bundles over $\widetilde{Y}_X$ being regarded as Clausen-Scholze analytic space in \cite{CS1}, \cite{CS2} and \cite{CS3}. One then has the definition for equal characteristic situation. 
\end{definition}

\begin{definition}
We now define the $(\infty,1)$-prestack $\mathrm{Finiteproj}_{{Y}}(-)$ fibered over $\mathrm{AnimaRigidAn}^\mathrm{aff}_{\mathbb{Q}_p,\mathrm{ffqc}}$ to be the functor sending $X$ to $(\infty,1)$-groupoid of vector bundles over ${Y}_X$. The imperfect space we consider here is defined to be the Frobenius quotient of
\begin{align}
\bigcup_{I=[s,r]}\mathrm{Spec}^\mathrm{CS}({\Pi}^I_{\mathrm{toric}}\otimes^\blacksquare \mathcal{O}_X,{\Pi}_{\mathrm{toric}}^{I,+}\otimes^\blacksquare \mathcal{O}_X).	
\end{align}
One then has the definition for equal characteristic situation. 
\end{definition}

\begin{definition}
We now define the prestack $\mathrm{Finiteproj}^\mathrm{disc}_{\widetilde{Y}}(-)$ fibered over $\mathrm{RigidAn}_{\mathbb{Q}_p,\mathrm{ffqc}}$ to be the functor sending $X$ to groupoid of vector bundles over $\widetilde{Y}_X$. One then has the definition for equal characteristic situation. 
\end{definition}

\begin{definition}
We now define the prestack $\mathrm{Finiteproj}^\mathrm{disc}_{Y}(-)$ fibered over $\mathrm{RigidAn}_{\mathbb{Q}_p,\mathrm{ffqc}}$ to be the functor sending $X$ to groupoid of vector bundles over $Y_X$. The imperfect space we consider here is defined to be the Frobenius quotient of
\begin{align}
\bigcup_{I=[s,r]}\mathrm{Spec}^\mathrm{CS}({\Pi}^I_{\mathrm{toric}}\otimes^\blacksquare \mathcal{O}_X,{\Pi}_{\mathrm{toric}}^{I,+}\otimes^\blacksquare \mathcal{O}_X).	
\end{align}
One then has the definition for equal characteristic situation. 
\end{definition}

\begin{proposition}
$\mathrm{Finiteproj}^\mathrm{disc}_{Y}(-)$ is an Artin stack over $\mathrm{RigidAn}_{\mathbb{Q}_p,\mathrm{ffqc}}$. The parallel result in equal characteristic sitaution holds true as well.
\end{proposition}

\begin{proof}
One can form the quotient Artin stacks $[*/\bigcup_{I=[s,r]}\mathrm{Res}_{\Pi^I_\mathrm{toric}/\mathbb{Q}_p}\mathrm{GL}^\mathrm{rigidification}_{\Pi^I_\mathrm{toric}}]$ which gives the result.	
\end{proof}

\indent Then one defines the following Hecke stacks. Let $I$ be a finite set, and we consider a disjoint decomposition of $I$ into $k$-subsets:
\begin{align}
I=\bigcup_{i=1}^k I_i.	
\end{align}

\begin{definition}
We now define the $(\infty,1)$-prestack $H_{I,I_1,...,I_k}\mathrm{Finiteproj}_{\widetilde{Y}}(-)$ fibered over 
\begin{align}
\mathrm{AnimaRigidAn}^\mathrm{aff}_{\mathbb{Q}_p,\mathrm{ffqc}} 
\end{align}
to be the functor sending $X$ to $(\infty,1)$-groupoid of the objects:
\begin{align}
(M_0,M_1,...,M_k, \{p_i\}_{i\in I}, f_1,...,f_k)	
\end{align}
where each $M_i,i=1,...,k$ is a vector bundle over the base space $\widetilde{Y}_X$, with $I$ points $\{p_i\}_{i\in I}\subset \widetilde{Y}_X$, with for each $i=1,...,k$ the isomorphism:
\begin{align}
f_i: M_i|_{\widetilde{Y}_X-\cup_{j\in I_i}p_j}	\overset{\sim}{\longrightarrow} M_{i-1}|_{\widetilde{Y}_X-\cup_{j\in I_{i-1}}p_j}.
\end{align}
One then has the definition for equal characteristic situation. Following this in similar fashion directly one can define $H_{I,I_1,...,I_k}\mathrm{Finiteproj}_{{Y}}(-)$, $H_{I,I_1,...,I_k}\mathrm{Finiteproj}^\mathrm{disc}_{\widetilde{Y}}(-)$ with $H_{I,I_1,...,I_k}\mathrm{Finiteproj}^\mathrm{disc}_{Y}(-)$.
\end{definition}

\begin{definition}
We then define the moduli stacks of arithmetic shtukas. This is defined as the fiber product along the following two maps:
\begin{align}
H_{I,I_1,...,I_k}\mathrm{Finiteproj}_{\widetilde{Y}}(-)&\rightarrow \mathrm{Finiteproj}_{\widetilde{Y}}(-)\times \mathrm{Finiteproj}_{\widetilde{Y}}(-)	\\
(M_0,M_1,...,M_k, \{p_i\}_{i\in I}, f_1,...,f_k)&\mapsto(M_0,M_k)
\end{align}
and
\begin{align}
\mathrm{Finiteproj}_{\widetilde{Y}}(-)&\rightarrow \mathrm{Finiteproj}_{\widetilde{Y}}(-)\times \mathrm{Finiteproj}_{\widetilde{Y}}(-)\\
M&\mapsto (M, \mathrm{Fro}^{*}M).	
\end{align}
We use the notation $\mathrm{Sht}_{I,I_1,...,I_k}\mathrm{Finiteproj}_{\widetilde{Y}}(-)$ to denote this $(\infty,1)$-prestack. One then has the definition for equal characteristic situation. Following this in similar fashion directly one can define $\mathrm{Sht}_{I,I_1,...,I_k}\mathrm{Finiteproj}_{{Y}}(-)$, $\mathrm{Sht}_{I,I_1,...,I_k}\mathrm{Finiteproj}^\mathrm{disc}_{\widetilde{Y}}(-)$ with $\mathrm{Sht}_{I,I_1,...,I_k}\mathrm{Finiteproj}^\mathrm{disc}_{Y}(-)$.
\end{definition}

\begin{remark}
One also has the construction in direct parallel fashion the stacks in Kummer setting as in \cite{EG}.	
\end{remark}

\section{Variants}

We have many different types of variants of the constructions above. For instance we can replace the Robba rings above with the following:
\begin{itemize}
\item[A] (Extended Fargues-Fontaine Curves) $(\Pi^I_\mathrm{toric}[(\log(1+T))^{1/2}]\widehat{\otimes}_{\mathbb{Q}_p}E)\otimes^\blacksquare \mathcal{O}_X$ as in \cite{BS} where $E$ contains square roots of $\varphi(\pi)$ and $\gamma(\pi)$ as in \cite{BS}; 
\item[B] (Logarithmic Fargues-Fontaine Curves) $(\Pi^I_\mathrm{toric}\left<\log(\log(1+T))\right>)\otimes^\blacksquare \mathcal{O}_X$, the notation $\left<\right>$ means the Tate algebra over the variable $\log(\log(1+T))$, after \cite{Fon};
\item[C] (Product Fargues-Fontaine Curves) $(\Pi^I_\mathrm{toric}\widehat{\otimes}_{\mathbb{Q}_p}\Pi^{I'}_\mathrm{toric})\otimes^\blacksquare \mathcal{O}_X$ as in \cite{CKZ} and \cite{PZ}.
\end{itemize}

\begin{proposition}
The parallel statements for these variant situations hold true. 
\begin{itemize}
\item[$\square$] The prestack $\mathrm{Finiteproj}^{\mathrm{disc},\mathrm{Fil}}_{\mathrm{FF}^\mathrm{imper}_{K,\mathrm{toric}}}(-)$ fibered over $\mathrm{RigidAn}_{\mathbb{Q}_p,\mathrm{ffqc}}$ is an Artin stack. The parallel result in equal characteristic sitaution holds true as well.
\item[$\square$] The $(\infty,1)$-prestack $\mathrm{Finiteproj}^{\mathrm{Fil}}_{\mathrm{FF}^\mathrm{imper}_{K,\mathrm{toric}}}(-)$ fibered over $\mathrm{AnimaRigidAn}^\mathrm{aff}_{\mathbb{Q}_p,\mathrm{ffqc}}$ is an $(\infty,1)$-stack. The parallel result in equal characteristic sitaution holds true as well.
\end{itemize}
\end{proposition}

We have the corresponding spaces in the following as well:
\begin{align}
&H_{I,I_1,...,I_k}\mathrm{Finiteproj}_{{Y}}(-), H_{I,I_1,...,I_k}\mathrm{Finiteproj}^\mathrm{disc}_{\widetilde{Y}}(-), H_{I,I_1,...,I_k}\mathrm{Finiteproj}^\mathrm{disc}_{Y}(-),\\ 
&\mathrm{Sht}_{I,I_1,...,I_k}\mathrm{Finiteproj}_{{Y}}(-), \mathrm{Sht}_{I,I_1,...,I_k}\mathrm{Finiteproj}^\mathrm{disc}_{\widetilde{Y}}(-) \mathrm{Sht}_{I,I_1,...,I_k}\mathrm{Finiteproj}^\mathrm{disc}_{Y}(-).
\end{align}

\newpage
\section*{Acknowledgements}
We thank Professor Kedlaya for conversation on the different perspectives  of Hodge-Iwasawa Theory. We mentioned the corresponding aspects discussed in this paper in our paper \cite{T2}. The author thanks Eugen Hellmann for the discussion around the construction of the rigid analytic Artin moduli stacks in the cyclotomic tower situation. The work \cite{EGH} contains the construction of such analytic stack when we consider the corresponding filtration as in \cite{He1}. We thank Professor Sorensen for the suggestions on generalizing the work of \cite{BS}.

\end{document}